\documentclass[11pt]{article}

\usepackage{bbold}
\usepackage{amsmath,amsthm,amssymb,amsfonts}
\usepackage{bm}
\usepackage{colonequals}
\usepackage{fullpage}
\usepackage{dsfont}
\usepackage{graphicx}
\usepackage{xcolor}
\usepackage{hyperref}
\usepackage{comment}
\usepackage{tikz-cd}
\usepackage{enumitem}

\renewcommand{\emptyset}{\varnothing}

\newcommand{\conn}{\mathrm{conn}}

\newcommand{\Ex}{\mathop{\mathbf{E}}}

\newcommand{\sI}{\mathcal{I}}
\newcommand{\sN}{\mathcal{N}}

\newcommand{\sM}{\mathcal{M}}

\newcommand{\One}{\bm{1}}
\newcommand{\CC}{\mathbb{C}}
\newcommand{\EE}{\mathbf{E}}

\newcommand{\RR}{\mathbb{R}}

\newcommand{\Sym}{\mathrm{Sym}}

\newcommand{\Part}{\mathrm{Part}}

\newcommand{\Tr}{\mathrm{Tr}}
\newcommand{\sym}{\mathrm{sym}}

\newtheorem{theorem}{Theorem}[section]
\newtheorem{remark}[theorem]{Remark}

\newtheorem{definition}[theorem]{Definition}

\newtheorem{proposition}[theorem]{Proposition}

\title{Gurau's spectral density is not a probability measure for individual real symmetric tensors}

\usepackage{authblk}

\author[1]{Maximilian Jerdee\thanks{Email: \texttt{mjerdee@umich.edu}.}}

\author[2]{Dmitriy Kunisky\thanks{Email: \texttt{kunisky@jhu.edu}.}}

\author[1]{Cristopher Moore\thanks{Email: \texttt{moore@santafe.edu}. Partially supported by the National Science Foundation through grant BIGDATA-1838251.}}

\affil[1]{Santa Fe Institute}

\affil[2]{Department of Applied Mathematics \& Statistics, Johns Hopkins University}

\date{March 6, 2026}

\begin{document}

\maketitle

\begin{abstract}
    Gurau (2020) proposed a generalization of the trace of the matrix resolvent to tensors of higher order, and recent work has explored analogs of the Wigner semicircle and Marchenko-Pastur distributions from random matrix theory as well as aspects of free probability theory from this perspective.
    In particular, when evaluated with appropriate large random tensors, the limiting expectations of the coefficients of a series expansion of Gurau's resolvent trace give the moment sequences of probability measures analogous to the above distributions.
    We construct, on the other hand, individual deterministic tensors such that the same coefficients evaluated on those tensors do \emph{not} give the moment sequence of any probability measure.
    Thus, the ``spectral density'' associated to Gurau's resolvent trace, while in a sense defined on average for certain random tensor ensembles, is not defined pointwise (unless perhaps as a signed measure) for all individual tensors.
\end{abstract}

\section{Introduction}

Recall that the classical resolvent trace of a real symmetric matrix $X \in \RR^{N \times N}_{\sym}$ is the function
\begin{equation}
R_X(z) \colonequals \frac{1}{N} \Tr(zI_N - X)^{-1},
\end{equation}
defined on $z \in \CC \setminus \RR$.
Expanding in a Laurent series, for such $z$ we have
\begin{equation}
R_X(z) = \sum_{k \geq 0} z^{-k - 1} \frac{1}{N}\Tr(X^k). \label{eq:resolvent-expansion}
\end{equation}
As is well-known, the coefficients $\frac{1}{N}\Tr(X^k)$ of this expansion are the moments of the empirical distribution of eigenvalues of $X$.

In \cite{Gurau-2020-WignerSemicircleLawTensors}, Gurau proposed a generalization of the resolvent trace of a matrix to real symmetric tensors.
We write $\Sym^p(\RR^N)$ for the set of $p$-ary real symmetric tensors with every axis having dimension $N$ (in particular $\Sym^2(\RR^N) = \RR^{N \times N}_{\sym}$); then, Gurau associated to each $T \in \Sym^p(\RR^N)$ a function $R_T: \CC \setminus \RR \to \CC$ playing the role of the resolvent trace.
We review the details of this proposal in Section~\ref{sec:gurau-resolvent} below.
Gurau's resolvent trace also admits a formal Laurent series expansion analogous to \eqref{eq:resolvent-expansion},
\begin{equation}
\label{eq:RT-exp}
R_T(z) = \sum_{k \geq 0} z^{-k - 1} \frac{1}{N}\sI_k(T),
\end{equation}
for certain homogeneous polynomials $\sI_k(T)$ in the entries of $T$ having degree $k$.
Accepting the analogy between this $R_T$ and the above $R_X$, it is natural to ask if there is a probability measure $\gamma_T$ on $\RR$ whose moments are given by these coefficients, i.e., having
\begin{equation}
\label{eq:muT}
\int x^k\,d\gamma_T(x) = \frac{1}{N} \sI_k(T) \text{ for each } k \geq 0.
\end{equation}
Equivalently, this asks whether $R_T(z)$ is the Stieltjes transform of some $\gamma_T$, but to avoid concerns about convergence of series let us focus on the moment formulation in \eqref{eq:muT}.
If there is such a $\gamma_T$, it ought to play for tensors some version of the role of the empirical distribution of eigenvalues of a matrix.
We call this object, when it exists, the \emph{Gurau spectral density}.

Encouragingly, the work \cite{Gurau-2020-WignerSemicircleLawTensors} found that, for a natural Gaussian distribution of $T$, the limiting \emph{expected} moments $\lim_{N \to \infty} \EE \frac{1}{N} \sI_k(T)$ are indeed a moment sequence, that of a natural tensorial generalization $\gamma = \gamma^{(p)}$ of the semicircle law $\gamma^{(2)}$ appearing prominently in random matrix theory.
Further, \cite{bonnin2024universality} showed the universality of this over more general random $T$ with entries independent up to symmetry.
The results of \cite{bonnin2024tensorial} include both concentration of the $\frac{1}{N} \sI_k(T)$ around their expectations as well as similar analysis for a tensorial generalization of the Marchenko-Pastur distribution and a notion of free convolution for these moment sequences and the associated measures.
However, to the best of our knowledge, previous work has not explored establishing whether, for each given $T \in \Sym^p(\RR^N)$, the Gurau spectral density $\gamma_T$ satisfying \eqref{eq:muT} exists.
This is also explicitly stated as an open problem in (the recently published version of) \cite{bonnin2024universality}.

Our goal is to show that, in fact, there are some deterministic real symmetric tensors for which the Gurau spectral density does not exist:
\begin{theorem}
    \label{thm:main}
    There exists $T \in \Sym^3(\RR^{27})$ for which $\sI_4(T) < 0$, and thus in particular for which there exists no probability measure $\gamma_T$ satisfying \eqref{eq:muT}.
\end{theorem}
\noindent
We note that, since $\sI_4(T)$ is a polynomial and in particular is continuous, if one such $T$ exists then so does an entire small ball of positive measure of such $T$ nearby.

This leaves the tensorial resolvent trace $R_T(z)$ in a curious position: on the one hand, the measure $\gamma^{(p)}$ studied by \cite{Gurau-2020-WignerSemicircleLawTensors,bonnin2024universality} whose moments are limits of expectations of $\frac{1}{N} \sI_k(T)$ is well-defined and is a natural generalization of the semicircle law (and likewise for the tensorial Marchenko-Pastur law of \cite{bonnin2024tensorial}).
On the other hand, Theorem~\ref{thm:main} shows that there is not a coherent notion of pointwise spectral density $\gamma_T$ associated to each realization of $T$ (at least that is compatible with the resolvent trace $R_T$), which, if it did exist, one would hope would converge in some weak sense to these limiting distributions like $\gamma^{(p)}$.
We hope Theorem~\ref{thm:main} will encourage future work to clarify this unusual situation.

\begin{remark}
    One possible ``fix'' to this state of affairs is to allow $\gamma_T$ to be a \emph{signed} measure. 
    While we show that for some tensors $T$ the sequence~$\frac{1}{N}\sI_k(T)$ cannot be the moments of a non-negative measure, this sequence can always be expressed as the moments of a signed Radon measure (as can any sequence of real numbers; see \cite{Boas-1939-MomentProblemSignedMeasure} and \cite[Theorem 3.11]{shohat1943problem}).
    Some constructions of signed measures have also been used before in the study of tensor eigenvalues and eigenvectors~\cite{sasakura2023signed,sasakura2024signed,regalado2024usefulness}.
    We thank Razvan Gurau for bringing this possibility to our attention and leave it as an interesting direction for further investigation.
\end{remark}

\section{Background}

We first review the construction of Gurau's resolvent trace $R_T(z)$, as well as giving a slightly different interpretation of the associated ``moments'' $\sI_k(T)$ from that in \cite{Gurau-2020-WignerSemicircleLawTensors,bonnin2024universality}, which will be useful in our construction for Theorem~\ref{thm:main}.

\subsection{Matchings and invariant polynomials}

\begin{definition}[Matchings]
    Write $\sM_{p,k}$ for the set of matchings of $[pk]$, where $[pk]$ is viewed as equipped with the partition $\pi_0 = \{\{1, \dots, p\}, \dots, \{(p - 1)k + 1, \dots, pk\}\}$ into subsets of size $p$.
    For $\mu \in \sM_{p, k}$, write $G(\mu)$ for the multigraph constructed by forming the perfect matching graph on $[pn]$ described by $\mu$ and then gluing the parts of $\pi_0$ together, leading to a $p$-regular multigraph on $n$ vertices which admit a natural labelling by $[n]$.
    Write $\sM_{p,k}^{\conn} \colonequals \{\mu \in \sM_{p, k}: G(\mu) \text{ is connected}\}$.
\end{definition}

\begin{definition}[Matching tensors]
    For~$\mu \in \sM_{p,k}$ viewed as a set of pairs of elements of $[pk]$, $\mu = \{\{a_1,b_1\},...,\{a_{pk/2},b_{pk/2}\}\}$, define the matching tensor~$w(\mu) \in (\RR^N)^{\otimes pk}$ as having, at index $i \in [N]^{pk}$, the value
    \begin{align*}
        w(\mu)_{i} = \prod_{\{a,b\} \in \mu} \One\{i_a = i_b\}.
    \end{align*} 
    For the empty matching, set $w(\emptyset) = 1$.
\end{definition}

\begin{definition}[Invariant polynomials]
    Suppose that $T \in \Sym^p(\RR^N)$.
    For each $\mu \in \sM_{p, k}$, the invariant tensor~$w(\mu)$ defines a polynomial
    \begin{align}
    M_{\mu}(T) 
    &\colonequals \langle T^{\otimes k}, w(\mu) \rangle \\
    &= \sum_{i: E(G(\mu)) \to [n]} \prod_{v \in V(G(\mu))} T_{i(\partial v)}, \label{eq:Mmu-graph}
    \end{align}
    where in the second formulation $i(\partial v)$ denotes the string of labels $i$ assigned to the edges incident with vertex $v$ of $G(\mu)$.
    For each $k \geq 0$, we define
    \begin{align*}
        M_k(T) &\colonequals \sum_{\mu \in \sM_{p, k}} M_{\mu}(T), \\
        M_k^{\conn}(T) &\colonequals \sum_{\mu \in \sM_{p,k}^{\conn}} M_{\mu}(T).
    \end{align*}
\end{definition}

\begin{remark}[Role in invariant theory]
    The matching tensors~$w(\mu)$ are invariant under the natural action of $O(N)$ on $(\RR^N)^{\otimes pk}$: for all $Q \in O(N)$,
    \[ Q^{\otimes pk} w(\mu) = w(\mu). \]
    By a well-known theorem due to Weyl~\cite{Weyl46}, the so-called \emph{first fundamental theorem} of the invariant theory of $O(N)$, these matching tensors in fact span all invariant tensors in~$(\RR^N)^{\otimes pk}$; their span is sometimes called the \emph{Brauer space} after the related work of \cite{brauer-37}.
    
    Analogously, the $M_{\mu}(T)$ span the homogeneous polynomials of degree $k$ that are invariant under the natural action of $O(N)$ on $\Sym^p(\RR^N)$.
    Because of the graphical interpretation in \eqref{eq:Mmu-graph} above, they are sometimes called \emph{tensor networks} associated to the graph $G(\mu)$.
    See \cite{KMW-2024-TensorCumulantsInvariantInference} for further exposition of this perspective and \cite{Procesi-2007-LieGroupsInvariantsRepresentations} for a modern reference on invariant theory.
\end{remark}

We note that different matchings $\mu$ may lead to isomorphic graphs $G(\mu)$, thus some of the $M_{\mu}(T)$ appear several times in the expressions for $M_k(T)$ and $M_k^{\conn}(T)$, and so effectively these have a greater ``weight'' in these sums.
The formalism preferred by \cite{Gurau-2020-WignerSemicircleLawTensors,bonnin2024universality,bonnin2024tensorial} and related work (see \cite{Gurau-2017-RandomTensors} for a book-length treatment) that indexes these polynomials by ``rooted maps'' rather than matchings leads to eliminating some of this double-counting.
We will see, however, that thinking in terms of matchings clarifies some calculations to come, so below we derive expressions for the $\sI_k(T)$ in terms of our $M_k^{\conn}(T)$.

\subsection{Gurau's resolvent trace}
\label{sec:gurau-resolvent}

The function $R_T(z)$ alluded to earlier is defined in \cite{Gurau-2020-WignerSemicircleLawTensors} as
\begin{align*}
    Z_T(z) &\colonequals \Ex_{g \sim \sN(0, I_N)} \exp\left(\frac{1}{pz}\langle T, g^{\otimes p}\rangle\right), \\
    R_T(z) &\colonequals \frac{1}{z \cdot Z_T(z)} \Ex_{g \sim \sN(0, I_N)} \frac{\|g\|^2}{N} \exp\left(\frac{1}{pz}\langle T, g^{\otimes p}\rangle\right),
    \intertext{and we also mention the extra definition}
    F_T(z) &\colonequals \log Z_T(z)
\end{align*}
that will be useful shortly.
Here $\sN(0, I_N)$ denotes the law of a standard Gaussian random vector in $\RR^N$.
We note that these expectations are not in general well-defined for $z \in \RR$; however, they clearly are well-defined for purely imaginary $z \in i\RR$, and can be continued analytically to $\CC \setminus \RR$ as discussed by \cite{Gurau-2020-WignerSemicircleLawTensors}.

In the language of statistical physics and spin glass theory, $Z_T(z)$ may be viewed as the \emph{partition function} of a \emph{Gaussian $p$-spin model}, similar to the perhaps better-known \emph{spherical $p$-spin model} where instead $g$ is taken uniform over a sphere of some radius.
In this interpretation, up to various notational conventions and normalizations, $z$ is the \emph{temperature}, $F_T(z)$ is the \emph{free energy}, and $R_T(z)$ is the sum of diagonal values of the \emph{two-point correlation function} or equivalently the sum of the second moments of the spins involved under the associated \emph{Gibbs measure}.

A calculation using Gaussian integration by parts gives the following equivalent representation that will be more convenient for our reasoning.
All derivatives here and below of the functions $R_T, F_T$, and $Z_T$ are with respect to $z$.
\begin{proposition}[Proposition 2.10 of \cite{bonnin2024universality}]
    \label{prop:RT-FT}
    For all $z \in \CC \setminus \RR$,
    \[ R_T(z) = z^{-1} - \frac{p}{N} F_T^{\prime}(z). \]
\end{proposition}

It will be useful to observe an alternative interpretation of the functions $Z_T$ and $F_T$ as moment and cumulant generating functions, respectively.
\begin{definition}
    Let $Y$ be a real-valued random variable.
    We write $m_k(Y) \colonequals \EE Y^k$ for the $k$th moment of $Y$, and the formal exponential generating function of the moments as
    \[ \phi_Y(z) \colonequals \sum_{k = 0}^{\infty} \frac{m_k(Y)}{k!}z^k. \]
    We write $\kappa_k(Y)$ for the $k$th cumulant of $Y$, defined implicitly by the relation of formal generating functions
    \[ \psi_Y(z) \colonequals \log \phi_Y(z) = \sum_{k = 1}^{\infty} \frac{\kappa_k(Y)}{k!} z^k. \]
    More directly, the $\kappa_k(Y)$ are defined by recursively inverting the relations
    \begin{equation}
    \label{eq:moment-cumulant}
    m_k(Y) = \sum_{\pi \in \Part([k])} \prod_{A \in \pi} \kappa_{|A|}(Y).
    \end{equation}
\end{definition}

Then, clearly by definition we have
\begin{align}
    Z_T(z) &= \phi_{\langle T, g^{\otimes p} \rangle}\left(\frac{1}{pz}\right), \label{eq:ZT-mgf} \\
    F_T(z) &= \psi_{\langle T, g^{\otimes p} \rangle}\left(\frac{1}{pz}\right), \label{eq:FT-cgf}
\end{align}
where $g \sim \sN(0, I_N)$.
The following describes the coefficients of these, the moments and cumulants of $\langle T, g^{\otimes p}\rangle$, in terms of the invariant polynomials defined previously:
\begin{proposition}
    \label{prop:Tg-moments-cumulants}
    Let $g \sim \sN(0, I_N)$.
    For any $T \in \Sym^p(\RR^N)$ and $k \geq 1$,
    \begin{align*}
        m_k(\langle T, g^{\otimes p}\rangle) &= M_k(T), \\
        \kappa_k(\langle T, g^{\otimes p}\rangle) &= M_k^{\conn}(T).
    \end{align*}
\end{proposition}
\begin{proof}
We can write the $k$th moment as 
\begin{align*}
    m_k(\langle T, g^{\otimes p}\rangle) 
    &= \Ex_g \langle T, g^{\otimes p} \rangle^k \\
    &=  \left\langle T^{\otimes k}, \Ex_g g^{\otimes pk } \right\rangle
    \intertext{The inner expectation may be evaluated by Wick's formula, which in our notation gives}
    &= \left\langle T^{\otimes k}, \sum_{\mu \in \sM_{p, k}} w(\mu) \right\rangle \\
    &= \sum_{\mu \in \sM_{p, k}} M_{\mu}(T) \\
    &= M_k(T).
\end{align*}
This establishes the first result. 
%See below for more discussion of the relationship between Wick's formula and such calculations.

For the second result, recall that to a matching $\mu$ we associate a graph $G(\mu)$.
This graph has a partition into connected components, to which is associated a partition of $\mu$ into smaller matchings.
Call this latter partition $\conn(\mu)$.
Then, we have
\begin{align*}
    M_{\mu}(T) = \prod_{\nu \in \conn(\mu)} M_{\nu}(T).
\end{align*}
Summing over all matchings, we then have that 
\begin{align*}
    M_k(T) &= \sum_{\mu \in \sM_{p,k}} M_{\mu}(T) \\
    &= \sum_{\mu \in \sM_{p,k}} \prod_{\nu \in \conn(\mu)} M_{\nu}(T)
\intertext{
If we collect matchings~$\mu$ by the partition~$\pi \in \Part([k])$ of vertices of the connected components of $G(\mu)$, within each subset~$A \in \pi$ of vertices we sum over all matchings $\nu$ of $p|A|$ objects such that $G(\nu)$ is a connected graph on~$|A|$ vertices, to obtain
}
    &= \sum_{\pi \in \Part([k])} \prod_{A \in \pi} M_{|A|}^{\conn}(T).
\end{align*}
It then follows from \eqref{eq:moment-cumulant} that
\[ \kappa_k(\langle T, g^{\otimes p}\rangle) = M^{\conn}_k(T). \qedhere \]
\end{proof}
\noindent 
We thus have the formal series in terms of invariant polynomials
\begin{align*}
Z_T(z) &= \sum_{k=0}^\infty \frac{M_k(T)}{k!} \left(\frac{1}{pz}\right)^k, \\
F_T(z) &= \sum_{k=1}^\infty \frac{M_k^\conn(T)}{k!} \left(\frac{1}{pz}\right)^k.
\end{align*}

We may now extract the coefficients of the formal series expansion \eqref{eq:RT-exp} of $R_T(z)$, which also specifies the $\sI_k(T)$ discussed above that are, up to scaling, the putative moments of the Gurau spectral density:
\begin{proposition}
    \label{prop:I-Mconn}
    For any $z_0 \in \CC \setminus \RR$ and $k \geq 1$,
    \begin{align}
    \frac{1}{N} \sI_k(T) 
    &\colonequals \frac{1}{(k + 1)!} \lim_{s \to \infty} \left(-z^2\right)^{k + 1} R^{(k + 1)}_T(sz_0) \label{eq:Ik-lim} \\
    &= \frac{1}{N} \frac{1}{p^{k-1}(k-1)!}M_k^{\conn}(T) \label{eq:Ik-Mconn} \\
    &= \frac{1}{N} \frac{1}{p^{k-1}(k-1)!}\kappa_k(\langle T, g^{\otimes p}\rangle).
    \end{align}
\end{proposition}
\noindent
The second expression above, on the right-hand side of \eqref{eq:Ik-lim}, is just the formula for extracting the coefficient of $z^{-k-1}$ in a series of the form \eqref{eq:RT-exp}; we note that this appears to be stated slightly incorrectly in the otherwise analogous Proposition 2.11 of~\cite{bonnin2024universality}.
The third expression, in \eqref{eq:Ik-Mconn}, describes the difference in scaling between our $M_k^{\conn}(T)$ and the $\sI_k(T)$ considered by \cite{Gurau-2020-WignerSemicircleLawTensors,bonnin2024universality} , mentioned above, which is due to their being defined in terms of different combinatorial structures.
See also Remark~\ref{rem:I-p2} below.
\begin{proof}
    We start from the expression in Proposition~\ref{prop:RT-FT}, which, combined with the observation \eqref{eq:FT-cgf} and Proposition~\ref{prop:Tg-moments-cumulants} gives
    \begin{align*}
        R_T(z)
        &= z^{-1} - \frac{p}{N} F_T^{\prime}(z) \\
        &= z^{-1} - \frac{p}{N} \frac{d}{dz}\left[\sum_{k = 1}^{\infty} \frac{M_k^{\conn}(T)}{k!} \left(\frac{1}{pz}\right)^k\right] \\
        &= z^{-1} + \frac{p}{N} \sum_{k = 1}^{\infty} \frac{M_k^{\conn}(T)}{p^k(k-1)!} z^{-k-1} \\
        &= z^{-1} + \sum_{k = 1}^{\infty} \frac{1}{N}\frac{M_k^{\conn}(T)}{p^{k-1}(k-1)!} z^{-k-1},
    \end{align*}
    and the stated result follows by matching coefficients.
\end{proof}

\begin{remark}
    \label{rem:I-p2}
    In the matrix case where $p = 2$, the only connected 2-regular graph on~$k$ vertices is the cycle~$C_k$, which is (isomorphic to) $G(\mu)$ for $2^{k - 1}(k - 1)!$ distinct matchings $\mu \in \sM_{2,k}$.
    For a matrix~$X \in \Sym^2(\RR^N) = \RR^{N \times N}_{\sym}$, we therefore have by Proposition~\ref{prop:I-Mconn} that
    \begin{equation}
    \label{eq:Ik-mx}
    \frac{1}{N} \sI_k(X) = \frac{1}{N} \Tr(X^k),
    \end{equation}
    and thus $R_X(z)$ in the tensorial definition applied to $X \in \Sym^2(\RR^N)$ indeed recovers the standard matrix resolvent trace.
\end{remark}

In summary, our calculations give the following reformulation of the question of whether for all $T \in \Sym^p(\RR^N)$ there exists a probability measure $\gamma_T$ satisfying the requirement \eqref{eq:muT} of the Gurau spectral density: if this were true, then, for any such $T$, the sequence 
\begin{align*}
    \alpha_0(T) &\colonequals 1, \\
    \alpha_k(T) &\colonequals \frac{1}{N} \frac{1}{p^{k - 1}(k - 1)!} \kappa_k(\langle T, g^{\otimes p}\rangle) \text{ for } k \geq 1
\end{align*}
would be the moment sequence of some probability measure.
Stated in these terms, perhaps the proposal looks less plausible, postulating that a certain transformation of many \emph{cumulant} sequences always gives valid \emph{moment} sequences.
Since cumulants do not share the easily described positivity constraints on moments, this seems unlikely to hold; indeed, it is quite surprising that this actually is true for the matrix case $p = 2$.

Our construction below will show that, in a sense, that reflects that the random variables $\langle T, g^{\otimes 2}\rangle = g^{\top}Tg$ for matrices $T$ are quite restricted in the laws they can have, while once $p \geq 3$ we may approximate many more random variables as $\langle T, g^{\otimes p}\rangle$ for suitable $T$.
One may view this as a consequence of the spectral theorem: for a matrix $T \in \Sym^2(\RR^N)$ with eigenvalues $\lambda_1, \dots, \lambda_N$, the law of $g^{\top} T g$ is that of $\sum_{i = 1}^N \lambda_i g_i^2$, a convolution of various rescalings and reflections of the $\chi^2(1)$ density.
There is no such structure theorem for tensors, and indeed we will see that we may therefore construct less well-behaved random variables as $\langle T, g^{\otimes 3}\rangle$.

\section{Construction of counterexample: Proof of Theorem~\ref{thm:main}}

\begin{proof}[Proof of Theorem~\ref{thm:main}]
    By the above reasoning, it suffices to construct a $T \in \Sym^p(\RR^N)$ such that $\kappa_4(\langle T, g^{\otimes p}\rangle) < 0$ when $g \sim \sN(0, I_N)$.
    Indeed, if so, then by Proposition~\ref{prop:I-Mconn} we will have $\sI_4(T) < 0$, giving the required contradiction.
    
    In addition to giving the counterexample itself, let us sketch how one might arrive at the idea of our construction.
    We will use repeatedly the formula for $\kappa_4$ of a general random variable $Y$,
    \[ \kappa_4(Y) = m_4(Y) - 4m_3(Y)m_1(Y) - 3m_2(Y)^2 + 12 m_2(Y)m_1(Y)^2 - 6 m_1(Y)^4, \]
    which for symmetric $Y$ reduces to
    \[ \kappa_4(Y) = m_4(Y) - 3m_2(Y)^2 = \EE[Y^4] - 3(\EE[Y^2])^2. \]
    Below we always take $g \sim \sN(0, I_N)$ for suitably large $N$.
    
    We start by finding a univariate polynomial $q$ such that $\kappa_4(q(g_1)) < 0$.
    In general this should be possible since such $q(g_1)$ can approximate in law \emph{any} absolutely continuous probability measure on $\RR$, by transforming $g_1$ by polynomial approximations of suitable cumulative distribution functions and their inverses.
    We will later approximately homogenize $q$ to form a random variable of the form $\langle T, g^{\otimes p}\rangle$, so the degree of $q$ will end up being the arity $p$ of our tensor $T$.
    One may check that no degree $p = 2$ polynomials $q$ have the above property; specifically, we have $\kappa_4(ag_1^2 + bg_1 + c) = 48(a^4 + a^2b^2)$, and this coincides with matrices $T$ leading to valid moment sequences $\frac{1}{N}\sI_k(T)$, per~\eqref{eq:Ik-mx}.
    But, it turns out that there is a simple example at the first tensorial degree~$p = 3$, 
    \[ \kappa_4\left(g_1 - \frac{1}{10}g_1^3\right) = -\frac{87}{250}. \]

    Next, we introduce $N$ additional $g_2, \dots, g_{N + 1}$ to approximate this $q(g_1)$ by a homogeneous $q(g_1, \dots, g_{N + 1})$.
    Another explicit calculation gives
    \[ \kappa_4\left(g_1\left(\frac{1}{N}\sum_{i = 2}^{N + 1}g_i^2\right) - \frac{1}{10}g_1^3\right) = -\frac{87}{250} + \frac{6}{N} + \frac{72}{N^2} + \frac{144}{N^3} \equalscolon f(N). \]
    That this takes the form $\kappa_4(g_1 - \frac{1}{10}g_1^3) + o(1)$ as $N \to \infty$ is as expected, because by the strong law of large numbers we have $\frac{1}{N}\sum_{i = 2}^{N + 1}g_i^2 \to 1$ as $N \to \infty$ almost surely.
    Since this is a homogeneous polynomial, there exists some $T^{(N)} \in \Sym^3(\RR^{N + 1})$ such that
    \[ \langle T^{(N)}, g^{\otimes 3}\rangle = g_1\left(\frac{1}{N}\sum_{i = 2}^{N + 1}g_i^2\right) - \frac{1}{10}g_1^3. \]
    We compute that $f(N) = \kappa_4(\langle T^{(N)}, g^{\otimes 3}\rangle) < 0$ once $N \geq 26$, giving the result.

    To be completely explicit, this tensor $T = T^{(26)}$ has entries
    \begin{align*}
        T_{1,1,1} &= -\frac{1}{10}, \\
        T_{1,i,i} = T_{i,1,i} = T_{i,i,1} &= \frac{1}{3 \cdot 26} = \frac{1}{78} \text{ for all } i \in \{2, 3, \dots, 27\},
    \end{align*}
    and all other entries zero.
    A general tensor $T^{(N)}$ representing the polynomial we construct for an arbitrary $N$ can be formed by replacing $26$ with $N$ and $27$ with $N + 1$ above.
\end{proof}

\section*{Acknowledgments}

Thanks to R\'{e}mi Bonnin, Charles Bordenave, and Razvan Gurau for helpful comments on a first draft of this note.

\bibliographystyle{alpha}
\bibliography{main}

\end{document}